\documentclass{amsart}
\usepackage{amsfonts}
\usepackage{amsmath}
\usepackage{amssymb}
\usepackage{amsthm}

\newtheorem{theorem}{Theorem}

\newtheorem{lemma}{Lemma}
\newtheorem{rem}{Remark}
\newtheorem{cor}{Corollary}

\begin{document}
\author{Mark Pankov}
\title[Characterization of semilinear isomorphisms and homeomorphisms ]
{Geometrical characterization of semilinear isomorphisms of vector spaces and 
semilinear homeomorphisms of normed spaces}
\address{Department of Mathematics and Computer Science, University of Warmia and Mazury,
S{\l}oneczna 54, 10-710 Olsztyn, Poland}
\email{pankov@matman.uwm.edu.pl, markpankovmath@gmail.com}

\maketitle

\begin{abstract}
Let $V$ and $V'$ be vector spaces over division rings
(possible infinite-dimensional) 
and let ${\mathcal P}(V)$ and ${\mathcal P}(V')$ be the associated projective spaces.
We say that $f:{\mathcal P}(V)\to {\mathcal P}(V')$ is a PGL-{\it mapping}
if for every $h\in {\rm PGL}(V)$ there exists $h'\in {\rm PGL}(V')$ such that $fh=h'f$.
We show that for every PGL-bijection the inverse mapping is a semicollineation.
Also, we obtain an analogue of this result for the projective spaces associated to normed spaces.
\end{abstract}

\section{Introduction}
Let $V$ and $V'$ be vector spaces over division rings.
Denote by ${\mathcal P}(V)$ and ${\mathcal P}(V')$ the associated projective spaces
(the dimensions of our vector spaces are assumed to be not less than $3$,
possible the vector spaces are infinite-dimensional).
We say that $f:{\mathcal P}(V)\to {\mathcal P}(V')$ is a PGL-{\it mapping}
if for every $h\in {\rm PGL}(V)$ there exists $h'\in {\rm PGL}(V')$ such that 
$$fh=h'f$$
(we identify ${\rm PGL}(V)$ with the group of all transformation of ${\mathcal P}(V)$
induced by linear automorphisms of $V$). 
Every strong semilinear embedding 
(a semilinear injection transferring any collection of linearly independent vectors to linearly independent vectors)
induces a PGL-mapping.
By \cite[Theorem 2]{Pankov}, every non-constant PGL-mapping of 
${\mathcal P}(V)$ to ${\mathcal P}(V')$ is induced by a strong semilinear embedding of $V$ in $V'$
if $V'$ is a vector space over a field and the image is contained in a subspace of $V'$
whose dimension is not greater than $\dim V$.
This characterization of strong semilinear embeddings is closely related 
to the description of so-called $l$-rigid isometric embeddings of Grassmann graphs
\cite[Section 2.10]{Pankov-b2}.
There are non-constant PGL-mappings which are not induced
by semilinear mappings \cite[Example 6]{Pankov}. 

In this note we show that for every PGL-bijection the inverse mapping is a semicollineation
(Theorem \ref{th1}).
As a consequence, we get a characterization of collineations 
in terms of the projective linear groups 
(Corollary \ref{cor1}).

At the end, we replace the projective linear groups by the groups of transformations
induced by linear automorphisms of  normed spaces and 
consider the corresponding analogues of PGL-mappings.
Our result (Theorem \ref{th2}) is in spirit of the classical theorem on isomorphisms of
the lattices formed by closed subspaces of normed spaces \cite{FL,M}.

\section{Linear algebra and projective geometry}
Let $V$ and $V'$ be left vector spaces over division rings $R$ and $R'$, respectively.
We suppose that the dimensions of $V$ and $V'$ both
are not less than $3$ (possible the vector spaces are infinite-dimensional).

A mapping $l:V\to V'$ is called {\it semilinear} if
$$l(x+y)=l(x)+l(y)\;\;\;\;\;\forall\;x,y\in V$$
and there is a homomorphism $\sigma:R\to R'$ such that
$$l(ax)=\sigma(a)l(x)\;\;\;\;\;\forall\;a\in R,x\in V.$$
If $R=R'$ and $\sigma$ is identity then the mapping $l$ is linear.
A  semilinear bijection of $V$ to $V'$ is a {\it semilinear isomorphism} if 
the associated homomorphism of $R$ to $R'$ is an isomorphism
(there are semilinear bijections over non-surjective homomorphisms of division rings).

Denote by ${\mathcal P}(V)$ the projective  space associated to $V$
(the points are $1$-dimensional subspaces of $V$ and 
the lines are defined by $2$-dimensional subspaces of $V$, i.e.
the line corresponding to a $2$-dimensional subspace $S$
consists of all $1$-dimensional subspaces contained in $S$).
Every semilinear injection $l:V\to V'$ induces the mapping
$$\pi(l):{\mathcal P}(V)\to {\mathcal P}(V')$$
which transfers every $1$-dimensional subspace $P\subset V$
to the $1$-dimensional subspace of $V'$ containing $l(P)$.
For every non-zero scalar $a\in R'$ we have $\pi(al)=\pi(l)$.
Conversely, if $l(V)$ contains two linearly independent vectors and
$s:V\to V'$ is a semilinear injection such that $\pi(l)=\pi(s)$ 
then $s$ is a scalar multiple of $l$.
This means that the projective linear group ${\rm PGL}(V)$ can be identified with 
the group of all transformations of ${\mathcal P}(V)$ induced by linear automorphisms of $V$,
i.e. the transformations of type $\pi(u)$, where $u\in {\rm GL}(V)$.

Three distinct points in a projective space are called {\it collinear} if there is a line containing them; 
otherwise, these points are said to be {\it non-collinear}.
A bijection between two projective spaces 
is a {\it semicollineation} if it transfers any triple of collinear points to a triple of collinear points
(in this case, lines go to subsets of lines).  
A semicollineation is a {\it collineation} if the inverse mapping also is a semicollineation,
or equivalently, three points are collinear if and only if their images are collinear.
By the Fundamental Theorem of Projective Geometry,
every collineation of ${\mathcal P}(V)$ to ${\mathcal P}(V')$ is induced by 
a semilinear isomorphism of $V$ to $V'$.
In the case when $\dim V\le \dim V'<\infty$,
every semicollineation of ${\mathcal P}(V)$ to ${\mathcal P}(V')$ is a collineation.
There is a semicollineation of the projective space associated to a $5$-dimensional vector space
to a non-desarguesian projective plane \cite{C}.
The following problem is still open: 
are there semicollineations of ${\mathcal P}(V)$ to ${\mathcal P}(V')$
which are not collineations?
On the other hand, it follows from Faure--Fr\"{o}licher-Havlicek's
version of the Fundamental Theorem of Projective Geometry \cite{FaureFrolicher,Havlicek}
that every semicollineation of ${\mathcal P}(V)$ to ${\mathcal P}(V')$
is induced by a semilinear injection of $V$ to $V'$
which sends any pair of linearly independent vectors to linearly independent vectors. 

We say that ${\mathcal X}\subset {\mathcal P}(V)$ is a PGL-{\it subset}
if every permutation on ${\mathcal X}$ can be extended to an element of 
${\rm PGL}(V)$.
Recall that $P_{1},\dots,P_{k}\in {\mathcal P}(V)$ form an {\it independent subset}
if non-zero vectors $x_{1}\in P_{1},\dots, x_{k}\in P_{k}$ are linearly independent.
A $(k+1)$-element subset ${\mathcal X}\subset {\mathcal P}(V)$ is called a $k$-{\it simplex} if
it is not independent and every $k$-element subset of ${\mathcal X}$ is independent.
Independent subsets and simplices are PGL-subsets.
In the case when $R$ is a field, 
all finite PGL-subsets of ${\mathcal P}(V)$ are described \cite[Proposition 2]{Pankov}.
We do not use this description here.

\section{PGL-bijections}
Recall that $f:{\mathcal P}(V)\to {\mathcal P}(V')$ is called 
a PGL-{\it mapping} if  it satisfies the following condition: 
\begin{enumerate}
\item[] for every $h\in {\rm PGL}(V)$ there is $h'\in {\rm PGL}(V')$ such that $fh=h'f$.
\end{enumerate}
All non-constant PGL-mappings are injective \cite[Lemma 12]{Pankov}.
In the case when $f$ is bijective, the latter condition can be rewritten as follows:
\begin{enumerate}
\item[]$fhf^{-1}$ belongs to ${\rm PGL}(V')$
for every $h\in {\rm PGL}(V)$.
\end{enumerate}

\begin{theorem}\label{th1}
If $f:{\mathcal P}(V)\to {\mathcal P}(V')$ is a {\rm PGL}-bijection
then $f^{-1}$ is a semicollinea\-tion.
\end{theorem}

\begin{cor}\label{cor1}
For every bijection $f:{\mathcal P}(V)\to {\mathcal P}(V')$ 
the following two conditions are equivalent:
\begin{enumerate}
\item[{\rm (1)}] $f$ is a collineation, 
\item[{\rm (2)}] $f$ and $f^{-1}$ both are {\rm PGL}-mappings.
\end{enumerate}
\end{cor}

The implication $(1)\Longrightarrow (2)$ follows from 
the Fundamental Theorem of Projective Geometry. 
The implication $(2)\Longrightarrow (1)$ is a consequence of Theorem \ref{th1}.

\begin{rem}{\rm
Every PGL-bijection $f:{\mathcal P}(V)\to {\mathcal P}(V')$ 
induces a monomorphism of ${\rm PGL}(V)$ to ${\rm PGL}(V')$
which is an isomorphism if $f^{-1}$ is a PGL-mapping.
If our vector spaces are finite-dimensional then
all isomorphisms between ${\rm PGL}(V)$ and ${\rm PGL}(V')$ are known \cite[Section IV.6]{Die}
and Corollary \ref{cor1} can be obtained from their description.
In the infinite-dimensional case, 
isomorphisms between the projective  linear groups are not described
(see \cite[Section 6]{R} for some remarks).
}\end{rem}

\begin{lemma}\label{lemma1-1}
Every {\rm PGL}-bijection $f:{\mathcal P}(V)\to {\mathcal P}(V')$ transfers
{\rm PGL}-subsets to {\rm PGL}-subsets.
\end{lemma}

\begin{proof}
Let ${\mathcal X}$ be a PGL-subset of ${\mathcal P}(V)$.
If $s'$ is a permutation  on $f({\mathcal X})$ then 
$$s:=f^{-1}s'f|_{\mathcal X}$$ 
is a permutation  on ${\mathcal X}$.
Let $h\in {\rm PGL}(V)$ be an extension of $s$.
Then $fhf^{-1}$ is an element of ${\rm PGL}(V')$ whose restriction to $f({\mathcal X})$ 
coincides with $s'$. Thus $f({\mathcal X})$ is a PGL-subset.
\end{proof}

\begin{lemma}\label{lemma1-2}
If $P_{1},P_{2},P_{3},P_{3}\in {\mathcal P}(V)$
form a {\rm PGL}-subset and $P_{1},P_{2},P_{3}$ belong to a line
then this line contains $P_{4}$.
\end{lemma}

\begin{proof}
Let $S$ be the $2$-dimensional subspace of $V$
containing $P_{1},P_{2},P_{3}$.
Consider the permutation on our PGL-subset which leaves fixed 
$P_{1},P_{2}$ and transposes $P_{3}$ and $P_{4}$.
Let $h\in {\rm PGL}(V)$ be an extension of this transposition.
Then $h$ transfers ${\mathcal P}(S)$ to itself.
Since $P_{3}$ is contained in $S$,
we have $P_{4}=h(P_{3})\subset S$.
\end{proof}

Theorem \ref{th1} is a consequence of the following.

\begin{lemma}\label{lemma1-3}
Every {\rm PGL}-bijection $f:{\mathcal P}(V)\to {\mathcal P}(V')$
transfers any triple of non-collinear points to a triple of non-collinear points.
\end{lemma}

\begin{proof}[Proof for the case when $\dim V\ge 4$]
Let $P_{1},P_{2},P_{3}\in {\mathcal P}(V)$ be a triple of non-collinear points.
Then
$$S:=P_{1}+P_{2}+P_{3}$$
is $3$-dimensional.
Suppose that the points $f(P_{1}),f(P_{2}),f(P_{3})$ are collinear, i.e.
$$S':=f(P_{1})+f(P_{2})+f(P_{3})$$
is $2$-dimensional.

Let us take any $Q\in {\mathcal P}(V)$ which is not contained in $S$
(such points exist, since $\dim V\ge 4$).
Then $P_{1},P_{2},P_{3},Q$ form an independent subset.
This is a PGL-subset and, by Lemma \ref{lemma1-1},
the same holds for the subset consisting of 
$$f(P_{1}),f(P_{2}),f(P_{3}),f(Q).$$
Since every $f(P_{i})$ is contained in $S'$, Lemma \ref{lemma1-2} guarantees
that $f(Q)\subset S'$.

Now consider any $P\in {\mathcal P}(S)$ distinct from $P_{1},P_{2},P_{3}$.
We choose $P_{i}$ and $P_{j}$ such that $P_{i},P_{j},P$ are non-collinear.
If $Q\in {\mathcal P}(V)$ is not contained in $S$ then 
$P_{i},P_{j},P,Q$ form an independent subset.
This is a PGL-subset and Lemma \ref{lemma1-1} implies 
that 
$$f(P_{i}),f(P_{j}),f(P),f(Q)$$
form a PGL-subset.
The subspace $S'$ contains $f(P_{i}),f(P_{j}),f(Q)$
and, by Lemma \ref{lemma1-2},
we have $f(P)\subset S'$.

So, the image of $f$ is contained in ${\mathcal P}(S')$.
This is impossible, since $f$ is bijective.
\end{proof}

\begin{proof}[Proof for the case when $\dim V=3$]
Let $P_{1},P_{2},P_{3}\in {\mathcal P}(V)$ be a triple of non-col\-linear points.
Suppose that the points $f(P_{1}),f(P_{2}),f(P_{3})$ are collinear 
and denote by $S'$ the $2$-dimensional subspace of $V'$ containing each $f(P_{i})$.
We distinguish the following two cases:
\begin{enumerate}
\item[(1)] Every line of ${\mathcal P}(V)$ contains precisely $3$ points.
Then ${\mathcal P}(V)$ is Fano plane and the same holds for ${\mathcal P}(V')$,
since $f$ is bijective.
\item[(2)] Every line of ${\mathcal P}(V)$ contains more than $3$ points.
\end{enumerate}
Consider any $P\in {\mathcal P}(V)$ such that
$P_{1},P_{2},P_{3},P$ form a $3$-simplex. 
Since every $3$-simplex is a PGL-subset, 
$$f(P_{1}),f(P_{2}),f(P_{3}),f(P)$$
form a PGL-subset by Lemma \ref{lemma1-1}.
Each $f(P_{i})$ is contained in $S'$ and it follows from Lemma \ref{lemma1-2} that
$f(P)\subset S'$. 
In the case (1), this means that $f$ is not  injective and we get a contradiction.

Consider the case (2).
It was established above that $f(P)\subset S'$
for every point $P\in {\mathcal P}(V)$ such that $P_{1},P_{2},P_{3},P$ form a $3$-simplex.
Suppose that  $Q\in {\mathcal P}(V)$ is contained in $P_{1}+P_{2}$
(the case when $Q$ is contained in $P_{1}+P_{3}$ or in $P_{2}+P_{3}$ is similar).
We choose $P\in {\mathcal P}(V)$
such that 
$$\{P_{1},P_{2},P_{3},P\}\;\mbox{ and }\;\{Q,P_{2},P_{3},P\}$$ 
are $3$-simplices (this is possible, since each line contains more than $3$ points).
Then $f(P)\subset S'$ and, by Lemma \ref{lemma1-1},
$$f(Q),f(P_{2}),f(P_{3}),f(P)$$
form a PGL-subset.
Lemma \ref{lemma1-2} implies that $f(Q)\subset S'$.
Thus the image of $f$ is contained in ${\mathcal P}(S')$ which is impossible.
\end{proof}

\section{Application to normed spaces}
Consider two normed space ${\mathcal N}=(N, |\cdot|)$ and ${\mathcal N}'=(N', |\cdot|')$,
where $N$ and $N'$ are vector spaces over the same field ${\mathbb R}$ or ${\mathbb C}$ 
and  $|\cdot|,|\cdot|'$ are norms.

The group of all linear automorphisms of ${\mathcal N}$ will be denoted by ${\rm GL}({\mathcal N})$.
In the case when $N$ is infinite-dimensional,
this is a proper subgroup of ${\rm GL}(N)$. 
We write ${\rm PGL}({\mathcal N})$
for the group of all transformations of ${\mathcal P}(N)$ induced by elements of ${\rm GL}({\mathcal N})$, 
i.e. ${\rm PGL}({\mathcal N})$ consists of all $\pi(u)$ such that $u\in {\rm GL}({\mathcal N})$.

We say that $f:{\mathcal P}(N)\to {\mathcal P}(N')$ is a ${\rm PGL}^{c}$-{\it mapping} if 
it satisfies the following condition:
\begin{enumerate}
\item[] for every $h\in {\rm PGL}({\mathcal N})$ there is $h'\in {\rm PGL}({\mathcal N}')$ such that $fh=h'f$.
\end{enumerate}
As in \cite[Lemma 12]{Pankov},
we can show that all non-constant ${\rm PGL}^{c}$-mappings are injective.
In the case when $f$ is bijective, the latter condition can be rewritten as follows:
\begin{enumerate}
\item[]$fhf^{-1}$ belongs to ${\rm PGL}({\mathcal N}')$
for every $h\in {\rm PGL}({\mathcal N})$.
\end{enumerate}

\begin{theorem}\label{th2}
 If the normed spaces ${\mathcal N}$ and ${\mathcal N}'$ both are infinite-dimensional then
for every bijection $f:{\mathcal P}(N)\to {\mathcal P}(N')$ 
the following two conditions are equivalent:
\begin{enumerate}
\item[{\rm (1)}] $f$ is induced by a semilinear homeomorphism of ${\mathcal N}$ to ${\mathcal N}'$,
\item[{\rm (2)}] $f$ and $f^{-1}$ both are ${\rm PGL}^{c}$-mappings.
\end{enumerate}
\end{theorem}

\begin{rem}\label{rem2}{\rm
The automorphism group of the field ${\mathbb R}$ is trivial 
and every semilinear mapping between real vector spaces is linear.
The field ${\mathbb C}$ has infinitely many automorphisms,
but an automorphism of ${\mathbb C}$ is continuous if and only if 
it is trivial or the complex conjugate mapping.
Therefore, every continuous semilinear mapping between
complex normed spaces is linear or conjugate linear.
}\end{rem}

\begin{rem}{\rm
If ${\mathcal N}$ and ${\mathcal N}'$ are finite-dimensional complex normed spaces
then the above statement fails.  
In this case, ${\rm PGL}({\mathcal N})$ and ${\rm PGL}({\mathcal N}')$
coincide with ${\rm PGL}(N)$ and ${\rm PGL}(N')$, respectively, 
and every PGL-bijection of ${\mathcal P}(N)$ to ${\mathcal P}(N')$ is a ${\rm PGL}^{c}$-bijection.
By Remark \ref{rem2}, there exist non-continuous semilinear isomorphisms between complex vector spaces. 
}\end{rem}

The implication $(1)\Longrightarrow (2)$ is obvious
and we prove $(2)\Longrightarrow (1)$.

\begin{lemma}\label{lemma2-1}
If ${\mathcal X}$ is a finite {\rm PGL}-subset of ${\mathcal P}(N)$
then every permutation on ${\mathcal X}$  can be extended to an element of 
${\rm PGL}({\mathcal N})$.
 \end{lemma}

\begin{proof}
Let $S$ be the minimal subspace of $N$ containing every element of ${\mathcal X}$.
This subspace is finite-dimensional, since ${\mathcal X}$ is a finite set.
Also, $S$ is invariant for all extensions of the permutations on ${\mathcal X}$.
The required statement follows from the fact that every linear automorphism of $S$
can be extended to an automorphism of the normed space ${\mathcal N}$.  
\end{proof}

Using Lemma \ref{lemma2-1}, we prove the following.

\begin{lemma}\label{lemma2-2}
Every ${\rm PGL}^{c}$-bijection of ${\mathcal P}(N)$ to ${\mathcal P}(N')$ transfers
finite {\rm PGL}-subsets to {\rm PGL}-subsets.
\end{lemma}

Let $f:{\mathcal P}(N)\to {\mathcal P}(N')$ be a bijection 
such that $f$ and $f^{-1}$ both are ${\rm PGL}^{c}$-mappings.
As in the proof of Lemma \ref{lemma1-3},
we establish that $f$ and $f^{-1}$ are semicollineations.
Hence $f$ is a collineation, i.e.
it is induced by a semilinear isomorphism $l:N\to N'$.

\begin{lemma}\label{lemma2-3}
If $l:N\to N'$ is a semilinear isomorphism 
such that $l$ and $l^{-1}$ both transfer closed subspaces of codimension $1$
to closed subspaces of codimension $1$
then $l$ is a homeomorphism of ${\mathcal N}$ to ${\mathcal N}'$. 
\end{lemma}

\begin{proof}
We refer \cite[Lemma B]{M} for the real case.
In the complex case, $l$ is linear or conjugate linear \cite[Lemma 2]{KM}
and the statement follows from \cite[Lemma 3]{FL}.
\end{proof}

Let $S$ be a closed subspace of ${\mathcal N}$  whose codimension is equal to $1$.
The condimension of $S':=l(S)$ is equal to $1$.
If $S'$ is not closed in ${\mathcal N}'$
then the closure  $\overline{S'}$ coincides with $N'$.
Consider any non-identity $h\in {\rm PGL}({\mathcal N})$
whose restriction to ${\mathcal P}(S)$ is identity.
Then $h':=fhf^{-1}$ is a non-identity element of ${\rm PGL}({\mathcal N}')$
whose restriction to ${\mathcal P}(S')$ is identity.
If $u'\in {\rm GL}({\mathcal N}')$ and $h'=\pi(u')$
then the restriction of $u'$ to $S'$ is a scalar multiple of ${\rm id}_{S'}$.
Since $\overline{S'}=N'$, this means that $u'$ is a scalar multiple of ${\rm id}_{N'}$ and
$h'$ is identity. The latter is impossible.
Hence $S'$ is closed.

Similarly, we show that 
$l^{-1}$ transfers closed subspaces of codimension $1$ to closed subspaces of codimension $1$.
Lemma \ref{lemma2-3} gives the claim.

\end{document}